\documentclass[uplatex,a4paper,reqno]{amsart}

\usepackage{tikz-cd}
\usepackage{amsmath}
\usepackage{amssymb}
\usepackage{enumerate}
\usepackage{ascmac}
\usepackage{mathrsfs}
\usepackage{multirow}
\usepackage{amsthm}
\usepackage{wrapfig}
\usepackage{here}
\usepackage{multirow}
\usepackage{hyperref}
\usepackage{graphicx,color}

\theoremstyle{plain}
    \newtheorem{thm}{Theorem}[section]
    \newtheorem{lem}[thm]{Lemma}
    \newtheorem{prop}[thm]{Proposition}

\theoremstyle{definition}    
    \newtheorem{defn}[thm]{Definition}
    
    \newtheorem{rem}[thm]{Remark}

\def\an{{\mathrm{an}}}
\def\aru{{\arrow[u,"\rotatebox{90}{$\in$}",phantom]}}

\def\B{{\mathcal{B}}}
\def\C{{\mathbb{C}}}
\def\Cc{{\mathcal{C}}}

\def\CH{{\mathrm{CH}}}

\def\Dc{{\mathscr{D}}}
\def\dec{{\mathrm{dec}}}
\def\del{\partial}

\def\div{{\mathrm{div}}}

\def\ev{{\mathrm{ev}}}

\def\id{{\mathrm{id}}}
\def\ind{{\mathrm{ind}}}
\def\J{{\mathcal{J}}}

\def\lra{\:{\longrightarrow}\:}
\def\NS{{\mathrm{NS}}}
\def\O{{\mathcal{O}}}

\def\P{{\mathbb{P}}}
\def\Pc{{\mathcal{P}}}
\def\Q{{\mathbb{Q}}}
\def\Qc{{\mathcal{Q}}}

\def\R{{\mathbb{R}}}

\def\rank{{\mathrm{rank}\:}}

\def\Spec{{\mathrm{Spec}\,}}
\def\tr{{\mathrm{tr}}}
\def\X{{\mathcal{X}}}

\def\Z{{\mathbb{Z}}}

\begin{document}
\title{Higher Chow cycles on $K3$ surfaces attached to plane quartics}
\author{Ken Sato}
\address{Department of Mathematics, Tokyo Institute of Technology}
\email{sato.k.da@m.titech.ac.jp}
\begin{abstract}
In this paper, we give an explicit construction of higher Chow cycles of type $(2,1)$ on $K3$ surfaces obtained as quadruple coverings of the projective plane ramified along smooth quartics.
The construction uses a pair of bitangents of the quartics.
We prove that the higher Chow cycles generate a rank 2 subgroup in the indecomposable part of the higher Chow group for very general members, by using a specialization argument and an explicit computation of the regulator map.
\end{abstract}
\maketitle
\section{Introduction}
Higher Chow groups $\CH^p(X,q)$ of smooth varieties $X$ introduced by Bloch are a generalization of the classical Chow groups and are related to many important invariants in algebraic geometry, $K$-theory, and number theory.
However, their explicit structures are still mysterious for almost all varieties when the codimension $p$ is greater than 1.
In this paper, we are interested in the first non-classical case $\CH^2(X,1)$ for $K3$ surfaces $X$.
The \textit{decomposable part} of $\CH^2(X,1)$ is the subgroup generated by cycles obtained as intersection products of lower codimensional cycles.
Since this is an easily accessible part, we are interested in the quotient of $\CH^2(X,1)$ by the decomposable part. 
This quotient is called the \textit{indecomposable part} and is denoted by $\CH^2(X,1)_\ind$.

Explicit constructions of non-trivial elements in $\CH^2(X,1)_\ind$ have been done for some lattice-polarized families of $K3$ surfaces (see the papers listed in the introduction of \cite{MS23}).
In \cite{MS23}, we construct families of higher Chow cycles on families of $K3$ surfaces $X$ with non-symplectic involutions $\iota$ in a systematic way.
The presence of involutions plays crucial roles in the construction of cycles, and also in the degeneration argument there.

In this paper, we consider the case of non-symplectic automorphisms of order $4$.
The most famous example, discovered by Kond\=o (\cite{Kondo}), is the the degree 4 covers of $\P^2$ ramified along smooth quartics.
They are naturally parametrized by the moduli space of smooth plane quartics.
After a suitable base change, we have a family $(\X,\sigma)\rightarrow S$ of $K3$ surfaces with non-symplectic automorphisms of order 4.
We explicitly construct a family of higher Chow cycles $\xi=\{\xi_s\}_{s\in S}$ on this family.
Our construction uses a pair $(l_0,l_1)$ of bitangents of the plane quartics.
On the $K3$ surfaces $X$ associated with the plane quartics, the inverse images of $l_0,l_1$ produce two rational curves intersecting at 2 points.
This configuration of rational curves enables us to construct a higher Chow cycle on $X$.
The main theorem of this paper is the following.
\begin{thm}\label{mainthmintro}
\rm{(Theorem \ref{mainthm})}
For a very general\footnote{In this paper, a very general point means that a point outside a union of countably many certain proper analytic subsets.} $s\in S$, the images of $\xi_s$ and $\sigma_*(\xi_s)$ in $\CH^2(\X_s,1)_\ind$ generate a subgroup of rank 2.
In particular, $\CH^2(\X_s,1)_\ind$ has rank $\ge 2$ for such $s\in S$.
\end{thm}
All previous constructions of higher Chow cycles on $K3$ surfaces have been done for lattice-polarized families of $K3$ surfaces, thus they are parametrized by orthogonal modular varieties.
On the other hand, our family of $K3$ surfaces is \textit{not} a lattice-polarized family, and is parametrized by a quotient of a complex ball by a unitary group, as proved by Kond\=o (\cite{Kondo}).
We give the first example of indecomposable higher Chow cycles on a non-lattice-polarized family of $K3$ surfaces.

The proof is divided into two parts; a specialization argument for a variant of normal functions and an explicit calculation on the normal function associated with our cycles.

We define the Jacobian $J(X,\sigma)$ for a $K3$ surface $X$ with a non-symplectic automorphism $\sigma$ of order $4$.
Then we have the corresponding Abel-Jacobi map
\begin{equation}\label{varAJ}
\begin{tikzcd}
\nu_{\sqrt{-1}}:&[-30pt]\CH^2(X,1)\arrow[r] & J(X,\sigma).
\end{tikzcd}
\end{equation}
On the family $(\X,\sigma)\rightarrow S$, the Jacobians $J(\X_s,\sigma_s)$ form a family of generalized complex tori over $S$, and the images of $\xi_s$ under the map (\ref{varAJ}) form a section of this Jacobian fibration, namely a certain type of normal functions.

The advantages of considering $J(X,\sigma)$ are its stability under deformation and its coincidence with the transcendental part of the Jacobian for very general members.
This enables us to reduce the proof of Theorem \ref{mainthmintro} to finding \textit{one} special member $(\X_s,\sigma_s)$ of the family on which the image of $\xi_s$ under the map (\ref{varAJ}) is non-torsion.
This approach is similar to the one in \cite{MS23}, but slightly different from it.
In \cite{MS23}, the degeneration to the boundary of the moduli space is used, but in this paper, we consider specialization inside the moduli space.

To find such a special member, we consider a 1-dimensional family of special quartics containing the Fermat curve $F$.
For each member $C$ of this family, $C$ has a $\mu_4$-action, and the associated $K3$ surface $X$ admits a rational map 
\begin{equation}\label{ratintro}
\begin{tikzcd}
\varphi:&[-30pt]C \times F \arrow[r, dashrightarrow] & X,
\end{tikzcd}
\end{equation}
which induces a birational equivalence between $(C\times F)/\mu_4$ and $X$.
This result is analogous to the fact that the Fermat surface of degree $m$ is birationally equivalent to the quotient of the self-product of Fermat curves of degree $m$ by a $\mu_m$-action (\cite{SK79}).

Using the rational map (\ref{ratintro}), we can prove that the period functions of this special family satisfy the Gauss hypergeometric differential equation of type $(1/2,1/2,1)$.
On the other hand, we can show that the images of $\xi_s$ under the transcendental regulator map are represented by a multivalued function $G(\lambda)$ satisfying the differential equation
\begin{equation}\label{inhomogPF}
\lambda(1-\lambda)\dfrac{d^2G}{d\lambda^2}+(1-2\lambda)\frac{dG}{d\lambda}-\frac{1}{4}G=\frac{1}{2-\lambda}\sqrt{-\frac{2}{\lambda}}.
\end{equation}
This implies that the images of $\xi_s$ under the map (\ref{varAJ}) are non-torsion for very general members of this special family, so we can prove the theorem.
The differential equation (\ref{inhomogPF}) is an example of the inhomogeneous Picard-Fuchs equations arising from the regulator map, which are studied in \cite{dAMS08}, \cite{CDKL16}.
Our function $G(\lambda)$ is a new variant of hypergeometric functions which appears as the value of the regulator map.

The structure of this paper is as follows.
In Section 2, we recall basic results on higher Chow cycles and normal functions.
In Section 3, we define the Jacobian $J(X,\sigma)$, recall the Kond\=o's construction of $K3$ surfaces attached to plane quartics, and prove the specialization argument for such $(X,\sigma)$.
In Section 4, we define the family $(\X,\sigma)\rightarrow S$ and construct the family of $(2,1)$-cycles on it.
In Section 5, we define the special 1-dimensional subfamily and the rational map (\ref{ratintro}).
Then, by explicitly computing the transcendental regulator map, we show that the images of the $(2,1)$-cycles under the map (\ref{varAJ}) are non-torsion for a very general member of this family.

\subsection*{Acknowledgement}
The author is very grateful to Shohei Ma for many valuable suggestions and encouragement.
This work was supported by JSPS KAKENHI 21H00971.

\section{Preliminaries}
\subsection{Higher Chow cycles}
For a smooth variety $X$ over $\C$, let $\CH^p(X,q)$ be the higher Chow group defined by Bloch (\cite{bloch}).
In this paper, we consider the case $(p,q)=(2,1)$.
An element of $\CH^2(X,1)$ is called a \textit{$(2,1)$-cycle}.
The higher Chow group $\mathrm{CH}^2(X,1)$ is isomorphic to the middle homology group of the following complex. For the proof, see, e.g., \cite{MS2} Corollary 5.3.
\begin{equation*}
K_2^{\mathrm{M}} (\C(X)) \xrightarrow{T} \displaystyle\bigoplus_{C\in X^{(1)}}\C(C)^\times \xrightarrow{\mathrm{div}}\displaystyle\bigoplus_{p\in X^{(2)}}\Z\cdot p 
\end{equation*}
Here $X^{(r)}$ denotes the set of all irreducible closed subsets of $X$ of codimension $r$. 
The map $T$ denotes the tame symbol map from the Milnor $K_2$-group of the function field $\C(X)$.
Therefore, each $(2,1)$-cycle is represented by a formal sum
\begin{equation}\label{formalsum}
\sum_j (C_j, f_j)\in \displaystyle\bigoplus_{C\in X^{(1)}}\C(C)^\times
\end{equation}
where $C_j$ are prime divisors on $X$ and $f_j\in \C(C_j)^\times$ are non-zero rational functions on them such that $\sum_j {\mathrm{div}}_{C_j}(f_j) = 0$ as codimension 2 cycles on $X$.

The intersection product on higher Chow groups induces the following map.
\begin{equation}\label{intersectionproduct}
\mathrm{Pic}(X)\otimes_\Z \Gamma(X,\mathcal{O}_X^\times)=\CH^1(X)\otimes_\Z \CH^1(X,1) \longrightarrow \mathrm{CH}^2(X,1)
\end{equation}
The image of this map is called the \textit{decomposable part} of $\mathrm{CH}^2(X,1)$ and is denoted by $\CH^2(X,1)_\dec$.
A \textit{decomposable cycle} is an element of $\CH^2(X,1)_\dec$.
The quotient $\CH^2(X,1)/\CH^2(X,1)_\dec$ is called the \textit{indecomposable part} of $\CH^2(X,1)$ and is denoted by $\mathrm{CH}^2(X,1)_{\mathrm{ind}}$. 
For a $(2,1)$-cycle $\xi$, $\xi_\ind$ denotes its image in $\mathrm{CH}^2(X,1)_{\mathrm{ind}}$. 

Let $C$ be a prime divisor on $X$, $\alpha\in \Gamma(X,\O_X^\times)$ and $[C]\in \mathrm{Pic}(X)$ be the class corresponding to $C$. 
The image of $[C]\otimes \alpha$ under (\ref{intersectionproduct}) is represented by $(C,\alpha|_C)$ in the presentation (\ref{formalsum}).

\subsection{The regulator map}
For a $\Z$-Hodge structure $H=(H_\Z,F^\bullet)$ of weight 2, the generalized intermediate Jacobian of $H$ is defined as the quotient group 
\begin{equation}
J(H) = \frac{H_\C}{H_\Z + F^2H_\C}.
\end{equation}
This is a generalized complex torus, i.e., a quotient of a $\C$-linear space by a discrete lattice.
Let $X$ be a $K3$ surface. 
We write $J(X) = J(H^2(X,\Z))$ and simply call it the \textit{Jacobian} of $X$.
The following \textit{Beilinson regulator map} plays a role of the Abel-Jacobi map in the study of $(2,1)$-cycles (cf. \cite{KLM}).
\begin{equation}\label{regulatormap}
\nu: \mathrm{CH}^2(X,1) \longrightarrow J(X)
\end{equation}
We recall a formula for the regulator map following \cite{Le} pp. 458--459.
By the intersection product, the Jacobian of $X$ is isomorphic to 
\begin{equation}\label{Delignefunctional}
J(X)\simeq \frac{(F^1H^2(X,\C))^\vee}{H_2(X,\Z)}
\end{equation}
where $(F^1H^2(X,\C))^\vee$ is the dual $\C$-linear space of $F^1H^2(X,\C)$ and we regard $H_2(X,\Z)$ as a subgroup of $(F^1H^2(X,\C))^\vee$ by integration. 

Let $\xi$ be a $(2,1)$-cycle represented by (\ref{formalsum}).
Let $D_j$ be the normalization of $C_j$ and $\mu_j: D_j\rightarrow X$ be the composition of $D_j\rightarrow C_j\hookrightarrow X$. 
We will define a topological 1-chain $\gamma_j$ on $D_j$.
If $f_j$ is constant, we define $\gamma_j = 0$.
If $f_j$ is not constant, we regard $f_j$ as a finite morphism from $D_j$ to $\P^1$. 
Then we define $\gamma_j$ as the pull-back of $[\infty, 0]$ by $f_j$, where $[\infty, 0]$ is a path on $\P^1$ from $\infty$ to $0$ along the positive real axis. 
By the condition $\sum_j \div_{C_j}(f_j) = 0$, $\gamma = \sum_j (\mu_j)_*\gamma_j$ is a topological 1-cycle on $X$.
Since $H_1(X, \Z) = 0$, there exists a 2-chain $\Gamma$ on $X$ such that $\partial\Gamma = \gamma$.
In this paper, $\Gamma$ is referred to as a \textit{2-chain associated with $\xi$}.
Then the image of $\xi$ under the regulator map is represented by the following element in $(F^1H^2(X,\C))^\vee$.
\begin{equation}\label{Levineformula}
\begin{tikzcd}
F^1H^2(X,\C)\ni \text{[}\omega\text{]} \arrow[r,mapsto] & \displaystyle\int_\Gamma\omega  + \sum_j\dfrac{1}{2\pi\sqrt{-1}}\displaystyle\int_{D_j-\gamma_j}\log (f_j)\mu_j^*\omega
\end{tikzcd}
\end{equation}
where $\log(f_j)$ is the pull-back of a logarithmic function on $\P^1-[\infty,0]$ by $f_j$.

From the formula, for a prime divisor $C$ on $X$ and $\alpha\in \C^\times$, the image of the $(2,1)$-cycle represented by $(C,\alpha)$ under the regulator map is $[C]\otimes \frac{1}{2\pi \sqrt{-1}}\log(\alpha)$.
In particular, we have 
\begin{equation}\label{JNS}
\nu(\CH^2(X,1)_\dec) = \NS(X)\otimes_\Z(\C/\Z) \quad (\subset  J(X)).
\end{equation}

\subsection{Normal functions}
We consider the regulator map in the relative setting.
Let $S$ be a smooth variety and $\mathcal H = (\mathcal H_\Z, F^\bullet)$ be a variation of $\Z$-Hodge structure of weight 2 over $S$.
Then 
\begin{equation*}
J(\mathcal H) = \frac{\mathcal H_\Z\otimes \O_S}{\mathcal H_\Z + F^2\mathcal H}
\end{equation*}
is a family of generalized complex tori over $S$.
Let $\pi: \X\rightarrow S$ be a smooth family of $K3$ surfaces over $S$.
Let $\J(\X)\rightarrow S$ be the family of Jacobians attached to the variation of Hodge structure $R^2\pi_*\Z_\X$.
The fiber of $\J(\X)\rightarrow S$ over $s\in S$ is $J(\X_s)$.

Suppose that we have irreducible divisors $\Cc_j$ on $\X$ which are smooth over $S$ and non-zero rational functions $f_j$ on $\Cc_j$ whose zeros and poles are also smooth over $S$.
Assume that they satisfy the condition $\sum_j\div_{(\Cc_j)_s}((f_j)_s) = 0$ on each fiber $\X_s$.
Then we have a family of $(2,1)$-cycles $\xi = \{\xi_s\}_{s\in S}$ such that $\xi_s\in \CH^2(\X_s,1)$ is represented by the formal sum $\sum_{j}((\Cc_j)_s,(f_j)_s)$.
In this paper, such a family of $(2,1)$-cycles is called an \textit{algebraic family of $(2,1)$-cycles}.
Then the section $S\rightarrow \J(\X); s\mapsto \nu(\xi_s)$ is holomorphic and satisfies the horizontality condition, namely it is a normal function (\cite{CDKL16} Remark 2.1). 
We denote this section by $\nu(\xi)$ and call it the \textit{normal function associated with $\xi$}.

\subsection{The transcendental regulator map}
In Section 5, we compute images of higher Chow cycles under the following variant of the regulator map.
\begin{equation}\label{transreg}
\begin{tikzcd}
\Phi:&[-30pt] \CH^2(X,1)\arrow[r,"\nu"] & J(X) \simeq \dfrac{(F^1H^2(X,\C))^\vee}{H_2(X,\Z)} \arrow[r,twoheadrightarrow] &\dfrac{(H^{2,0}(X))^\vee}{H_2(X,\Z)}
\end{tikzcd}
\end{equation}
where the last projection is induced by $H^{2,0}(X) \hookrightarrow F^1H^2(X,\C)$.
This map is called the \textit{transcendental regulator map}.

Let $\omega$ be a non-zero holomorphic 2-form on $X$.
Considering the paring with $[\omega]\in H^{2,0}(X)$, we have an isomorphism between the target of (\ref{transreg}) and the abelian group $\C/\Pc(\omega)$, where $\Pc(\omega)$ is the set of periods of $X$ with respect to $\omega$, i.e.,
\begin{equation*}
\Pc(\omega) = \left\{\displaystyle \int_{\Gamma}\omega  \in \C : \Gamma \text{ is a toplogical 2-cycle on }X.\right\}.
\end{equation*}
By the formula (\ref{Levineformula}), the transcendental regulator map is calculated as
\begin{equation}\label{transregval}
\Phi(\xi)([\omega]) \equiv  \int_{\Gamma}\omega  \mod \Pc(\omega)
\end{equation}
where $\Gamma$ is a 2-chain associated with $\xi$.

Let $\X\rightarrow S$ be a family of $K3$ surfaces and $\omega$ be a nowhere vanishing relative 2-form.
For an open subset $U$ of $S$ in the classical topology, a period function with respect to $\omega$ on $U$ is a holomorphic function on $U$ given by $U\ni s\mapsto \int_{\Gamma_s}\omega_s$, where $\{\Gamma_s\}_{s\in U}$ is a $C^\infty$-family of 2-cycles on $\X_s$.
The period functions with respect to $\omega$ generate (as sheaves of abelian groups) the subsheaf $\Pc_\omega$ of the sheaf $\O_S^\an$ of holomorphic functions. 
Let $\Qc_\omega$ be the quotient $\O_{S}^\an/\Pc_\omega$.
For each $s\in S$, the evaluation map $\O_{S}^\an \rightarrow \C; f\mapsto f(s)$ induces the map
\begin{equation*}
\ev_s: \Gamma(S,\Qc_\omega) \lra \C/\Pc(\omega).
\end{equation*}

For an algebraic family $\xi = \{\xi_s\}_{s\in S}$ of $(2,1)$-cycles, the right-hand side of (\ref{transregval}) varies holomorphically on $S$ (\cite{CL}, Proposition 4.1).
Thus we have the section $\nu_{\tr}(\xi)(\omega)\in \Gamma(S,\Qc_\omega)$ such that for each $s\in S$,
\begin{equation*}
\ev_s(\nu_{\tr}(\xi)(\omega)) = \Phi(\xi_s)([\omega_s]) \quad\text{in } \C/\Pc(\omega_s).
\end{equation*}
We use the following elementary property of sections of $\Qc_\omega$.
For the proof, see, e.g., \cite{Sat24}, Lemma 2.4.
\begin{lem}\label{importantlemma}
If $\varphi\in \Gamma(S,\Qc_\omega)$ is non-zero, $\ev_s(\varphi)$ is non-zero for a very general $s\in S$.
In particular, if $\nu_{\tr}(\xi)(\omega)\neq 0$, $\Phi(\xi_s)\neq 0$ for a very general $s\in S$.
\end{lem}

\section{$K3$ surfaces with non-symplectic automorphisms of order $4$}
Let $X$ be a $K3$ surface.
An automorphism $\sigma:X\rightarrow X$ of order $4$ is called \textit{purely non-symplectic} if $\sigma^*$ acts on $H^{2,0}(X)$ by a primitive 4th root of unity.
Hereafter we impose the normalization condition that $\sigma^*$ acts on $H^{2,0}(X)$ by $\sqrt{-1}$.

In Section 3.1, we define a variant of the Jacobian for a pair $(X,\sigma)$ as above.
In Section 3.2, we recall the Kond\=o's family (\cite{Kondo}), namely the pair $(X,\sigma)$ attached to plane quartics.
In Section 3.3, we prepare our specialization argument in advance.

\subsection{The variant of the normal function}
Let $(X,\sigma)$ be a pair of a $K3$ surface and a purely non-symplectic automorphism of order 4 on $X$.
We decompose $H^2(X,\C)$ into the eigenspaces of the $\sigma^*$-action as follows.
\begin{equation*}
H^2(X,\C) = H_{1}\oplus H_{-1}\oplus H_{\sqrt{-1}}\oplus H_{-\sqrt{-1}}.
\end{equation*}
We define the following primitive sublattices of $H^2(X,\Z)$.
\begin{equation*}
\begin{aligned}
&L(X,\sigma) = (H_{1}\oplus H_{-1})\cap H^2(X,\Z) \\
&H(X,\sigma) =\left(H_{\sqrt{-1}}\oplus H_{-\sqrt{-1}}\right)\cap H^2(X,\Z)
\end{aligned}
\end{equation*}
Note that they are orthogonal to each other in $H^2(X,\Z)$ with respect to the cup product.
Since $H(X,\sigma)$ is the kernel of the map $(\sigma^2)^*+\id:H^2(X)\rightarrow H^2(X)$, this is a sub $\Z$-Hodge structure of $H^2(X)$.
We define the \textit{Jacobian of $(X,\sigma)$} by
\begin{equation*}
J(X,\sigma) = J(H(X,\sigma)^\vee),
\end{equation*}
where $H(X,\sigma)^\vee$ is the dual of $H(X,\sigma)$.
Since $H(X,\sigma)^\vee$ has a $\Z[\sqrt{-1}]$-module structure by $\sigma^*$, we can regard $J(X,\sigma)$ as a $\Z[\sqrt{-1}]$-module.

The inclusion $H(X,\sigma)\hookrightarrow H^2(X,\Z)$ induces the map 
\begin{equation}\label{surjHodge}
\begin{tikzcd}
H^2(X,\Z)=H^2(X,\Z)^\vee \arrow[r] & H(X,\sigma)^\vee,
\end{tikzcd}
\end{equation}
where the first equality follows from the unimodularity of $H^2(X,\Z)$.
Since (\ref{surjHodge}) is surjective after tensoring with $\Q$, this induces the surjection $J(X)\twoheadrightarrow J(X,\sigma)$.
Then we define the variant of the regulator map by the composition
\begin{equation*}
\begin{tikzcd}
\nu_{\sqrt{-1}}:\CH^2(X,1) \arrow[r,"\nu"] & J(X) \arrow[r,twoheadrightarrow] &J(X,\sigma).
\end{tikzcd}
\end{equation*}
Since $\sigma$ is purely non-symplectic, we have the inclusion $H^{2,0}(X)\subset H(X,\sigma)_\C$.
Thus the surjection $J(X)\twoheadrightarrow H^{2,0}(X)^\vee/H_2(X,\Z)$ in (\ref{transreg}) factors $J(X,\sigma)$.
This implies the following.
\begin{prop}\label{reduce1}
For a $(2,1)$-cycle $\xi$, if $\Phi(\xi)$ is non-torsion, then $\nu_{\sqrt{-1}}(\xi)$ is so.
\end{prop}

Let $\pi: (\X,\sigma)\rightarrow S$ be a family of $K3$ surfaces with purely non-symplectic automorphisms of order 4, i.e., $\pi:\X\rightarrow S$ is a family of $K3$ surfaces and $\sigma: \X\rightarrow \X$ is an $S$-automorphism of order $4$ such that for every $s\in S$, $\sigma_s$ is purely non-symplectic.
We define the variation of Hodge structure $\mathcal H(\X,\sigma)$ as the kernel of $(\sigma^2)^*+\id: R^2\pi_*\Z_\X\rightarrow R^2\pi_*\Z_\X$ and $\mathcal H(\X,\sigma)^\vee$ as its dual.
We have the morphism $R^2\pi_*\Z_{\X}\rightarrow  \mathcal H(\X,\sigma)^\vee$ as in (\ref{surjHodge}).

Let $\J(\X,\sigma) = \J(\mathcal H(\X,\sigma)^\vee)$ be the family of generalized complex tori over $S$ associated with $\mathcal H(\X,\sigma)^\vee$.
The fiber of $\J(\X,\sigma)\rightarrow S$ over $s\in S$ is $J(\X_s,\sigma_s)$.
The morphism $R^2\pi_*\Z_{\X}\rightarrow  \mathcal H(\X,\sigma)^\vee$ induces the surjection 
\begin{equation}\label{surjJac}
\begin{tikzcd}
\J(X)\arrow[r,twoheadrightarrow] & \J(\X,\sigma).
\end{tikzcd}
\end{equation}
For an algebraic family of $(2,1)$-cycles $\xi = \{\xi_s\}_{s\in S}$,
\begin{equation*}
\begin{tikzcd}
S\arrow[r] & \J(\X,\sigma); s\arrow[r,mapsto] &\nu_{\sqrt{-1}}(\xi_s)
\end{tikzcd}
\end{equation*}
is the composition of $\nu(\xi):S\rightarrow \J(\X)$ and the projection (\ref{surjJac}), so this is holomorphic.
This section is denoted by $\nu_{\sqrt{-1}}(\xi)$.
By definition, we have the following.
\begin{prop}\label{sectionlem}
For an algebraic family of $(2,1)$-cycles $\xi = \{\xi_s\}_{s\in S}$, if there exists a point $s_0\in S$ such that $\nu_{\sqrt{-1}}(\xi_{s_0})$ is non-torsion, $\nu_{\sqrt{-1}}(\xi)$ is so. 
\end{prop}

\subsection{Kond\=o's family}
We recall the example discovered by Kond\=o (\cite{Kondo}).
Let $C$ be a smooth plane quartic.
Let $X\rightarrow \P^2$ be the degree 4 cover ramified along $C$.
Then $X$ is a $K3$ surface with the covering transformation $\sigma$.
Replacing $\sigma$ by $\sigma^{-1}$ if necessary, we may assume that $\sigma^*$ acts on $H^{2,0}(X)$ as $\sqrt{-1}$.

Since the quotient $X/\langle \sigma^2\rangle$ is a del Pezzo surface, 
\begin{equation*}
L(X,\sigma) = H^2(X,\Z)^{\sigma^2} = H^2(X/\langle \sigma^2\rangle,\Z)(2)
\end{equation*}
is generated by algebraic cycles.
Thus we have 
\begin{equation}\label{NScontain}
L(X,\sigma)\subset \NS(X).
\end{equation}
Let $\Lambda_{K3}$ be the $K3$ lattice $U^{\oplus 3}\oplus E_8^{\oplus 2}$ and $\alpha: H^2(X,\Z)\rightarrow \Lambda_{K3}$ be a marking, i.e., an isometry between lattices.
Let $\rho$ be the isometry on $\Lambda_{K3}$ defined by $\alpha\circ \sigma^*\circ \alpha^{-1}$.
We consider the eigenspace decomposition $(\Lambda_{K3})_\C=\Lambda_{1}\oplus \Lambda_{-1}\oplus \Lambda_{\sqrt{-1}}\oplus \Lambda_{-\sqrt{-1}}$ with respect to the $\rho$-action and define the sublattice $H$ of $\Lambda_{K3}$ by $\left(\Lambda_{\sqrt{-1}}\oplus \Lambda_{-\sqrt{-1}}\right)\cap \Lambda_{K3}$.
The lattice $H$ is isomorphic to $U\oplus U(2) \oplus D_4^{\oplus 2} \oplus A_1^{\oplus 2}$ (\cite{Kondo2} p. 2833).
In particular, the signature of $H$ is $(2,12)$.
The isometry $\alpha$ induces an isometry $H(X,\sigma)\xrightarrow{\:\sim\:} H$.

Let $\B$ be the 6-dimensional complex ball defined by
\begin{equation*}
\B = \{v\in \P(\Lambda_{\sqrt{-1}}):\langle v,\overline{v}\rangle > 0\}.
\end{equation*}
Since $H^{2,0}(X)$ is contained in $\Lambda_{\sqrt{-1}}$ by the convention, $\B$ can be regarded as the period domain.
Let $\Gamma$ be a discrete group
\begin{equation}\label{group}
\Gamma = \{\gamma \in O(H):\gamma\circ \rho = \rho\circ \gamma\}.
\end{equation}
If we regard $H$ as a $\Z[\sqrt{-1}]$-module by the $\rho$-action, the group $\Gamma$ can be identified with the unitary group over $\Z[\sqrt{-1}]$ with respect to the Hermitian form $h(x,y) = \sqrt{-1}\langle x,\rho(y)\rangle + \langle x, y\rangle$.

Since $\Lambda_{\sqrt{-1}}=H\otimes_{\Z[\sqrt{-1}]}\C$, the group $\Gamma$ naturally acts on $\B$.
This action is properly discontinuous, and the period map induces the isomorphism between the moduli space $\mathcal M$ of smooth plane quartics and a Zariski open subset of $\B/\Gamma$ (\cite{Kondo} Theorem 2.5). 
In particular, since the period map is dominant, we have the following.
\begin{prop}\label{NSprop}
For a very general $C\in \mathcal M$, the pair $(X,\sigma)$ attached to $C$ satisfies $L(X,\sigma)=\NS(X)$.
\end{prop}

\subsection{The specialization argument}
Let $(\X,\sigma)\rightarrow S$ be a family of $K3$ surfaces with purely non-symplectic automorphisms of order 4.
Suppose that we have a dominant morphism $p:S\rightarrow \mathcal M$ to the moduli space of smooth plane quartics such that, for any  $s\in S$, $(\X_s,\sigma_s)$ is isomorphic to the pair attached to the quartic $p(s)\in \mathcal M$.

\begin{prop}\label{nonzeroprop}
Let $\xi=\{\xi_s\}_{s\in S}$ be an algebraic family of $(2,1)$-cycles on $\X\rightarrow S$.
Suppose that $\nu_{\sqrt{-1}}(\xi)$ is non-torsion.
Then $(\xi_s)_\ind$ and $(\sigma_*(\xi_s))_\ind$ generate a rank 2 subgroup in $\CH^2(\X_s,1)_\ind$ for a very general $s\in S$.
\end{prop}
\begin{proof}
For $s\in S$, let $\Xi_s$ be the subgroup of $\CH^2(\X_s,1)$ generated by $\xi_s$ and $\sigma_*(\xi_s)$ and let $N_s$ be the subgroup of $J(\X_s, \sigma_s)$ generated by $\nu_{\sqrt{-1}}(\xi_s)$ and $\nu_{\sqrt{-1}}(\sigma_*(\xi_s))$.
By definition, we have $\nu_{\sqrt{-1}}(\Xi_s) = N_s$.

First, we show that $\nu_{\sqrt{-1}}$ factors $\CH^2(\X_s,1)_\ind$ for a very general $s\in S$.
To prove this, it suffices to show that $\nu_{\sqrt{-1}}$ annihilates the decomposable part for a very general $s\in S$.
By (\ref{JNS}), the image of $\CH^2(\X_s,1)_\dec$ under the regulator map is $\NS(\X_s)\otimes_\Z (\C/\Z)$.
By Proposition \ref{NSprop} and the assumption that $p: S\rightarrow \mathcal M$ is dominant, this coincides with $L(\X_s,\sigma_s)\otimes_\Z (\C/\Z)$ for a very general $s\in S$.
Since $L(\X_s,\sigma_s)$ and $H(\X_s,\sigma_s)$ are orthogonal to each other, $L(\X_s,\sigma_s)\otimes_\Z (\C/\Z)$ is mapped to $0$ by the projection $J(\X_s)\twoheadrightarrow J(\X_s,\sigma_s)$.
Thus $\nu_{\sqrt{-1}}$ annihilates $\CH^2(\X_s,1)_\dec$ for a very general $s\in S$ and we finish the first step.

Since $\nu_{\sqrt{-1}}(\Xi_s) = N_s$, we reduce the proof to showing that $N_s$ is rank 2 for a very general $s\in S$.
Let $N$ be the subgroup of $\Gamma(S,\J(\X,\sigma))$ generated by $\nu_{\sqrt{-1}}(\xi)$ and $\nu_{\sqrt{-1}}(\sigma_*(\xi))$.
Then for each $s\in S$, $N_s$ is the image of $N$ under the restriction map $\Gamma(S,\J(\X,\sigma))\rightarrow J(\X_s,\sigma_s)$.
Since $N$ is a set of countably many holomorphic sections, $N\twoheadrightarrow N_s$ is bijective for a very general $s\in S$.
Therefore, it is sufficient to prove that $N$ is of rank 2.

We have a $\Z[\sqrt{-1}]$-module structure on $\Gamma(S,\J(\X,\sigma))$ by the $\sigma^*$-action and $N$ is its sub-$\Z[\sqrt{-1}]$-module generated by $\nu_{\sqrt{-1}}(\xi)$.
For any positive integer $m$, $m\cdot\nu_{\sqrt{-1}}(\xi)$ does not vanish by the assumption. 
Thus $N$ is a free $\Z[\sqrt{-1}]$-module.
In particular, $N$ is of rank 2 as an abelian group and we complete the proof.
\end{proof}

\begin{rem}\label{abelianvar}
Since $H^{2,0}(\X_s)$ is contained in $H_{\sqrt{-1}}$ by the convention, the projection $H_{\sqrt{-1}}\oplus H_{-\sqrt{-1}}\twoheadrightarrow H_{-\sqrt{-1}}$ induces the surjective map
\begin{equation}\label{abrem}
\begin{tikzcd}
J(\X_s,\sigma_s) = \dfrac{H_{\sqrt{-1}}\oplus H_{-\sqrt{-1}}}{H(\X_s,\sigma_s)^\vee+H^{2,0}(\X_s)} \arrow[r,twoheadrightarrow] & 
\dfrac{H_{-\sqrt{-1}}}{H(\X_s,\sigma_s)^\vee} \simeq \dfrac{\Lambda_{-\sqrt{-1}}}{H^\vee}
\end{tikzcd} 
\end{equation}
between generalized complex tori.
By the $\rho$-action, $H^\vee$ can be regarded as a free $\Z[\sqrt{-1}]$-module of rank 7.
Then we have\footnote{The embedding $\tau$ is the conjugate of the natural embedding.} $\Lambda_{-\sqrt{-1}} \simeq H^\vee\otimes_{\Z[\sqrt{-1}],\tau}\C$ and the target of (\ref{abrem}) is isomorphic to the direct product of the elliptic curve $(\C/\Z[\sqrt{-1}])^{\oplus 7}$.
Furthermore, the fiber of (\ref{abrem}) is isomorphic to the 6-dimensional complex vector space $H_{\sqrt{-1}}/H^{2,0}(\X_s)$.
Therefore  $J(\X_s,\sigma_s)$ is isomorphic to a 6-dimensional holomorphic vector bundle over the 7-dimensional abelian variety as a complex manifold.
It seems to be a new phenomenon that a kind of the Jacobian of $K3$ surfaces has such a structure.
\end{rem}

\section{Higher Chow cycles attached to pairs of bitangents of quartics}
In this section, we define the family of $K3$ surfaces and construct a family of $(2,1)$-cycles on them.
The key to our construction of cycles is to find rational curves on $X/\sigma = \P^2$ which intersect with the ramification locus ($=$ a plane quartic) in a special way.
This construction is analogous to the constructions used in \cite{Sre14},  \cite{MS23} and \cite{Sre24}, where the degree 2 cover cases are considered.
In our case, bitangents of plane quartics satisfy the desired property.

\subsection{Construction of the $K3$ family} 
Let $f\in \C[Z_0,Z_1,Z_2]$ be a quartic form (a homogeneous polynomial of degree 4) such that the zero locus of $f$ in $\P^2$ is a smooth plane quartic $C$.
Let $X$ be the $K3$ surface defined by the equation $W^4=f(Z_0,Z_1,Z_2)$ in $\P^3$, where the homogeneous coordinates of $\P^3$ are $Z_0,Z_1,Z_2,W$.
The projection $\pi: X\rightarrow \P^2; [Z_0:Z_1:Z_2:W]\mapsto [Z_0:Z_1:Z_2]$ is a quadruple covering ramified along $C$, and the automorphism $\sigma:X\rightarrow X$ defined by 
\begin{equation*}
\sigma([Z_0:Z_1:Z_2:W]) = [Z_0:Z_1:Z_2:\sqrt{-1}W]
\end{equation*}
is a generator of the covering transformation group of $\pi$.

A \textit{bitangent} of a plane quartic $C$ is a line $l$ that is tangent to $C$ at every point in $C\cap l$.
It is well-known that $C$ has exactly 28 bitangents.
A bitangent $l$ of $C$ is called an \textit{inflection bitangent} if $l$ and $C$ intersect at one point. 
Let $l_0$ and $l_1$ be distinct bitangents of $C$ such that \textit{neither} is an inflection bitangent.
Let $P$ be the intersection point of $l_0$ and $l_1$ and $Q$ be a point in $\pi^{-1}(P)$.

\begin{defn}
A \textit{marked bitangent pair} is a quadruple $(f,l_0,l_1,Q)$ as above.
Let $S$ be an algebraic variety consisting of of all marked bitangent pairs.
We have the family $(\X,\sigma)\rightarrow S$ of $K3$ surfaces with purely non-symplectic automorphisms of order 4 and its section $Q$.
\end{defn}

For each  $(f,l_0,l_1,Q)\in S$, by assigning the quartic defined by $f$, we have a dominant (indeed surjective) morphism $S\rightarrow \mathcal{M}$ to the moduli space of smooth plane quartics.
The family $(\X,\sigma)\rightarrow S$ satisfies the condition at the beginning of Section 3.3.
The reason we consider the space of defining equations instead of plane quartics is to get a family of $K3$ surfaces $\X$ over the whole base space $S$.

\subsection{Construction of higher Chow cycles}
Let $s=(f,l_0,l_1,Q)\in S$ be a marked bitangent pair.
For $i=0,1$, the pull-back of $f$ by $l_i\hookrightarrow \P^2$ is a square of a homogeneous polynomial of degree 2.
Therefore the inverse image $\pi^{-1}(l_i)$ decomposes into 2 curves $l_{i,0}$ and $l_{i,1}$ such that $l_{i,j}\xrightarrow{\pi} l_i$ is the double covering branched along 2 points in $l_i\cap C$.
Furthermore, $l_{i,0}$ and $l_{i,1}$ are swapped by the automorphism $\sigma$.
For $i\in \Z/4\Z$, we define $Q_i = \sigma^i(Q)\in \pi^{-1}(P)$.
By changing the labels of $l_{i,j}$ if necessary, we may assume that $l_{0,0}$ and $l_{1,0}$ intersect at $Q_0$ and $Q_2$.
Then the configuration of the lines $l_{i,j}$ is as shown in Figure \ref{lfigure}.

\begin{figure}[h]
\centering
\includegraphics[width=10cm]{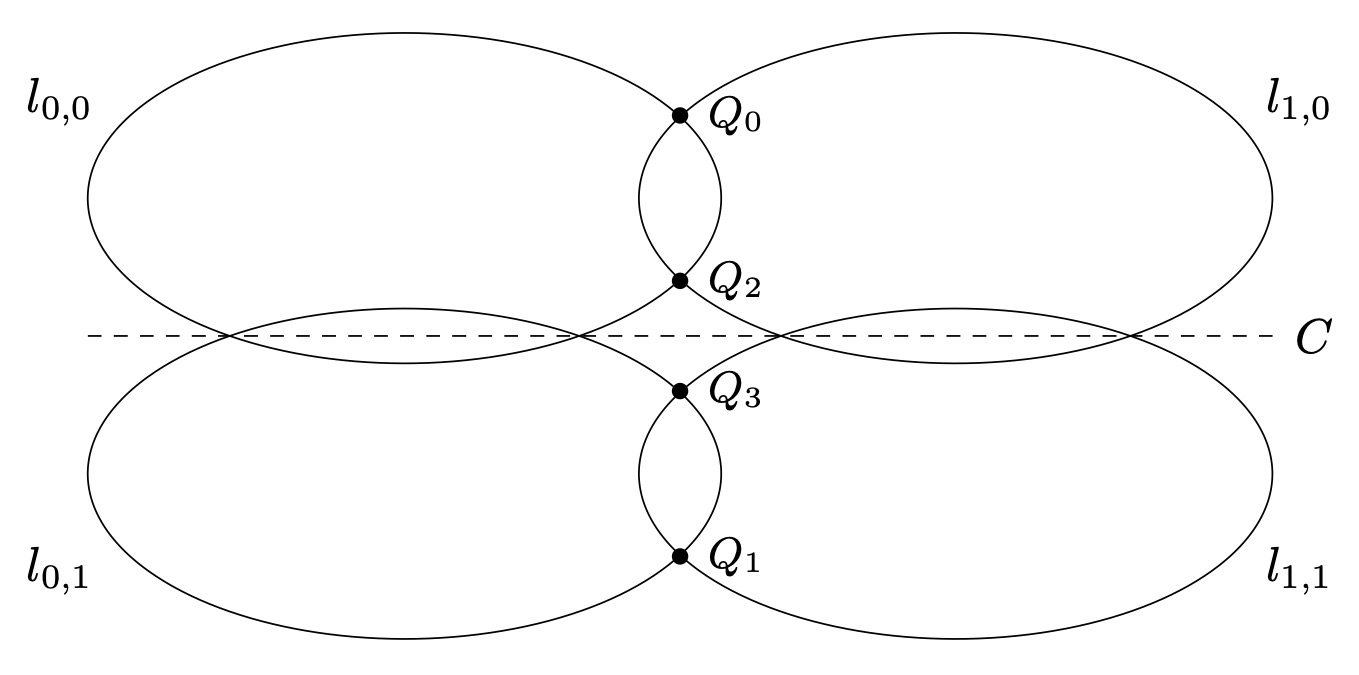}
\caption{The lines $l_{i,j}$ on $X$}
\label{lfigure}
\end{figure} 

Since $l_{0,0},l_{1,0}$ are rational curves, we can find rational functions $f_{0,0}\in \C(l_{0,0})^\times$ and $f_{1,0}\in \C(l_{1,0})^\times$ which satisfy
\begin{equation}\label{ratfunction1}
\div_{l_{0,0}}(f_{0,0}) = Q_2-Q_0 = -\div_{l_{1,0}}(f_{1,0}).
\end{equation}
\begin{defn}
Let $\xi_s$ be the $(2,1)$-cycle on $\X_s$ represented by the formal sum
\begin{equation}\label{cycle}
(l_{0,0},f_{0,0})+(l_{1,0},f_{1,0}).
\end{equation}
Note that we have freedom of the choices of $f_{i,0}$, but the ambiguity is only multiplication by some nonzero constant functions.
If we choose different rational functions satisfying (\ref{ratfunction1}), the resulting $(2,1)$-cycles differ by a decomposable cycle of the form $(l_{0,0},\alpha_{0})+(l_{1,0},\alpha_{1})$, where $\alpha_{0},\alpha_{1}\in\C^\times$.
In particular, $(\xi_{s})_\ind$ does not depend on the choice of rational functions.
Moreover, since $[l_{0,0}], [l_{1,0}]\in L(\X_s,\sigma_s)$, the ambiguity does not affect the image $\nu_{\sqrt{-1}}(\xi_s)\in J(\X_s,\sigma_s)$.
\end{defn}

The main theorem of this paper is as follows.
\begin{thm}\label{mainthm}
For a very general $s\in S$, $(\xi_s)_\ind$ and $(\sigma_*(\xi_s))_\ind$ generate a rank 2 subgroup in $\CH^2(\X_s,1)_\ind$.
In particular, $\rank \CH^2(\X_s,1)_\ind\ge 2$ for such a $s\in S$.
\end{thm}

Since $\sigma(l_{i,0})=l_{i,1}$, $\sigma_*(\xi_s)$ coincides with the $(2,1)$-cycle constructed from two rational curves $l_{0,1},l_{1,1}$ intersecting at 2 points $Q_1$ and $Q_3$. 

We will explain that Zariski locally on $S$, $\xi=\{\xi_s\}_{s\in S}$ can be regarded as an algebraic family.
First, we note that $\pi^{-1}(l_{i})\rightarrow S \:\:(i=0,1)$ decomposes into two irreducible components at each fiber, and we have a section $Q:S\rightarrow \pi^{-1}(l_{i})$ contained only one of these components, so $\pi^{-1}(l_{i})\rightarrow S$ decomposes into two $\P^1$-bundles $l_{i,0}\rightarrow S$ and $l_{i,1}\rightarrow S$, where they are labeled so that the section $Q$ is contained in $l_{i,0}$.

For $s\in S$, we can take a Zariski open neighborhood $U$ of $s$ such that $l_i|_U\rightarrow U$ is a trivial $\P^1$-bundle.
Then $l_{i,0}|_U\rightarrow U$ is a family of conics with a section $Q$, hence it also becomes a trivial $\P^1$-bundle over $U$ for $i=0,1$.
Therefore, we can find rational functions $f_{0,0}\in \C(l_{0,0})^\times$ and $f_{1,0}\in \C(l_{1,0})^\times$ such that their zeros and poles are $Q_0=Q$ and $Q_2=\sigma^2(Q)$ (hence smooth over $U$) and their restrictions to each fiber satisfy the relation (\ref{ratfunction1}).

Thus we have an algebraic family $\xi_{U} =\{\xi_{s}\}_{s\in U}$ of $(2,1)$-cycles on $\X|_U\rightarrow U$.
By gluing the normal functions $\nu_{\sqrt{-1}}(\xi_U)$, we have the section $\nu_{\sqrt{-1}}(\xi)$ of $\J(\X,\sigma)\rightarrow S$.
By Proposition \ref{nonzeroprop}, the proof of Theorem \ref{mainthm} is reduced to showing that $\nu_{\sqrt{-1}}(\xi)$ is non-torsion.

By Proposition \ref{sectionlem}, it suffices to find a point $s_0\in S$ such that $\nu_{\sqrt{-1}}(\xi_{s_0})$ is non-torsion.
Then it is enough to find a point $s_0\in S$ such that the image of $\xi_{s_0}$ under the transcendental regulator map is non-torsion by Proposition  \ref{reduce1}.
To compute $\Phi(\xi_{s_0})$ for certain $s_0$, we will define a special 1-dimensional subfamily $\X_T\rightarrow T$ of $\X\rightarrow S$ whose period functions can be computed explicitly. 

\section{The special family}
In this section, we define the special 1-dimensional subfamily $\X_T\rightarrow T$ and  prove that $\nu_{\sqrt{-1}}(\xi_{t})$ is non-torsion for a very general $t\in T$ by computing the image of $\xi_t$ under the transcendental regulator map.

\subsection{Special plane quartics}
For $t\in \C-\{0,-1/2\}$, we define the separable cubic polynomial $f_t(z)$ by 
\begin{equation*}
f_t(z) = (2z+t)(2z^2-2z-t) = z^4-(z^2-2z-t)^2.
\end{equation*}
Then we define the quartic form $g_t(Z_0,Z_1,Z_2)$ by
\begin{equation*}
g_t(Z_0,Z_1,Z_2) = Z_2^4 - Z_0^4f_t\left(\frac{Z_1}{Z_0}\right) = Z_2^4-Z_0\left(2Z_1+tZ_0\right)\left(2Z_1^2-2Z_1Z_0-tZ_0^2\right).
\end{equation*}
Let $C_t$ be the smooth plane quartic defined by $g_t$.
The lines $l_0:Z_2=Z_1$ and $l_1:Z_2=-Z_1$ are bitangents of $C_t$.
If $t\neq -1$, neither is an inflection bitangent.
Let $\X_t\subset \P^3$ be the $K3$ surface defined by $W^4=g_t(Z_0,Z_1,Z_2)$ and $\pi: \X_t\rightarrow \P^2$ be the quadruple covering.
Using the inhomogeneous coordinates $z_1 = Z_1/Z_0, z_2=Z_2/Z_0, w= W/Z_0$, the defining equation for $\X_t$ is given by 
\begin{equation*}
w^4 = z_2^4 - f_t(z_1).
\end{equation*}
The point $(z_1,z_2,w)=(0,0,\sqrt{t})\in \X_t$ is in the inverse image of $P = l_0\cap l_1$ by $\pi$.
We denote this point by $Q$.

Let $T=\Spec\C\left[t,\dfrac{1}{t(t+1)(2t+1)},\sqrt{t}\right]$.
By the construction above, we have a marked bitangent pair $(g_t,l_0,l_1,Q)$ for each $t\in T$, so we have a morphism $T\rightarrow S$.
In the following we consider the special 1-dimensional family $(\X_T,\sigma_T)\rightarrow T$ which is the base change of $(\X,\sigma)\rightarrow S$.

\subsection{A rational map}
We will define a rational map from a product of curves to $\X_t$.
This map is used to calculate the periods of $\X_t$.
Let $F$ be the Fermat quartic defined by $U_2^4=U_1^4-U_0^4$.
We define the rational map $\varphi$ by
\begin{equation*}
\begin{tikzcd}[row sep = tiny]
&\varphi:&[-70pt]C_t\times F \arrow[r, dashrightarrow] & \X_t\\
&&([V_0:V_1:V_2],[U_0:U_1:U_2])\arrow[r,mapsto] \aru&  \text{[}Z_0:Z_1:Z_2:W\text{]}=[U_0V_0:U_0V_1:U_1V_2:U_2V_2]\aru,
\end{tikzcd}
\end{equation*}
where the coordinates $[V_0:V_1:V_2]$ of $C_t$ satisfy $g_t(V_0,V_1,V_2)=0$.
Using the inhomogeneous coordinates, we can describe $\varphi$ as $(v_1,v_2,u_1,u_2)\mapsto (z_1,z_2,w)=(v_1,u_1v_2,u_2v_2)$, where $v_1 = V_1/V_0$, $v_2 = V_2/V_0$, $u_1 = U_1/U_0$ and $u_2 = U_2/U_0$.
The indeterminacy locus of $\varphi$ consists of 16 points defined by $U_0=V_2=0$. 
Let $b: Y_t\rightarrow C_t\times F$ be the blow up along these 16 points.
Then we have the morphism $\widetilde{\varphi}:Y_t\rightarrow X_t$ such that
\begin{equation*}
\begin{tikzcd}
Y_t \arrow[d,"b"']\arrow[dr,"\widetilde{\varphi}"]&\\
C_t\times F \arrow[r,dashrightarrow,"\varphi"'] & \X_t
\end{tikzcd}
\end{equation*}
commutes.
We will show that $\widetilde{\varphi}$ is generically $4:1$.
Consider the following $\mu_4$-action on $C_t\times F$.
\begin{equation*}
\begin{tikzcd}[row sep =tiny]
\mu_4\times (C_t\times F)\arrow[r]  & C_t\times F \\
(\zeta, [V_0:V_1:V_2],[U_0:U_1:U_2])\aru \arrow[r,mapsto]  & ([V_0:V_1:\zeta^{-1}V_2],[U_0:\zeta U_1: \zeta U_2])\aru
\end{tikzcd}
\end{equation*}
Then this $\mu_4$-action on $C_t\times F$ lifts to $Y_t$.
By calculating the rings of invariants under the $\mu_4$-action, we see that $\widetilde{\varphi}:Y_t\rightarrow X_t$ factors as follows.
\begin{equation}\label{gen41}
\begin{tikzcd}
Y_t \arrow[r,"\text{quotient by }\mu_4"]  &[50pt] Y_t/\mu_4 \arrow[r,"\text{blowing-up along}","\{W=Z_2=0\}\sqcup\{Z_0=Z_1=0\}"'] &[80pt] \X_t
\end{tikzcd}
\end{equation}

\begin{rem}
If $t = -1$, the curve $C_t$ is isomorphic to the Fermat quartic $F$ by the coordinate transformation $(v_1,v_2)\mapsto (u_1,u_2)= \left(v_1/(v_1-1),v_2/(v_1-1)\right)$.
In this case, the rational map $\varphi$ coincides with the rational map from $F\times F$ to the Fermat quartic surface defined in \cite{SK79} p. 98.
\end{rem}

\subsection{The Picard-Fuchs differential operator}
Let $\omega$ be the relative 2-form on $\X_T\rightarrow T$ defined by 
\begin{equation}
\omega = \dfrac{dz_1\wedge dz_2}{w^3}.
\end{equation}
We denote the restriction of $\omega$ to $\X_t$ by $\omega_t$.
In this section, we will find the differential operator $\Dc_t$ which annihilates the period functions with respect to $\omega$, namely, the Picard-Fuchs differential operator with respect to $\omega$.

Let $\theta_t$ be the 1-form on $C_t$ defined by $dv_1/v_2^2$ and $\mu$ be the 1-form on $F$ defined by $du_1/u_2^3$.
By the explicit description of $\varphi$, we have
\begin{equation}\label{formrelation}
b^*(pr_1^*(\theta_t)\wedge pr_2^*(\mu)) = \widetilde{\varphi}^*(\omega_t),
\end{equation}
where $pr_i$ denotes the $i$-th projection.
We will show that $\theta_t$ is the pull-back of a 1-form on an elliptic curve.
For $\lambda = -2t$, let $E_\lambda$ be the elliptic curve defined by 
\begin{equation*}
E_\lambda : y^2 = x(x-1)(x-\lambda).
\end{equation*}
Then we have the quadruple covering $\psi: C_t\rightarrow E_\lambda$ defined by 
\begin{equation}\label{psidef}
(v_1,v_2)\mapsto (x,y) = \left(\frac{2v_1^2}{2v_1+t},\frac{2v_2^2v_1\left(v_1+t\right)}{\left(2v_1+t\right)^2}\right).
\end{equation}
Let $\eta_\lambda$ be the 1-form on $E_\lambda$ defined by $dx/(2y)$.
Then we have
\begin{equation}\label{psipullback}
\psi^*(\eta_{\lambda}) = \theta_t.
\end{equation}

\begin{prop}\label{PFop}
Let $\Dc_t:\O_T^\an \rightarrow \O_T^\an$ be the differential operator defined by 
\begin{equation*}
\Dc_t = -\frac{1}{2}t(1+2t)\frac{d^2}{dt^2}-\frac{1}{2}(1+4t)\frac{d}{dt}-\frac{1}{4}.
\end{equation*}
Then every period function $p(t)$ with respect to $\omega$ satisfies $\Dc_t(p)=0$.
\end{prop}
\begin{proof}
We have the following morphism of $\Z$-Hodge structures.
\begin{equation*}
\begin{tikzcd}
\phi: H^2(C_t\times F)\arrow[r,"b^*"] &H^2(Y_t) \arrow[r,"\widetilde{\varphi}_!"] &H^2(\X_t)
\end{tikzcd}
\end{equation*}
By (\ref{gen41}), $\widetilde{\varphi}$ is generically $4:1$, so $\widetilde{\varphi}_{!}\circ \widetilde{\varphi}^*$ coincides with multiplication by $4$ (cf. \cite{Voi02}, Remark 7.29).
Hence we have $4[\omega_t] = \phi([pr_1^*(\theta_t)\wedge pr_2^*(\mu)])$ by the relation (\ref{formrelation}).
On the other hand, we have the following map between singular homologies.
\begin{equation}\label{singhom}
\begin{tikzcd}
H_2(\X_t,\Z) \arrow[r,"\phi^\vee"] & H_2(C_t\times F,\Z) \arrow[r,twoheadrightarrow] & H_1(C_t,\Z)\otimes_{\Z} H_1(F,\Z),
\end{tikzcd}
\end{equation}
where the first map is the dual of $\phi$ and the second map is the projection induced by the K\"unneth formula.
For a 2-cycle $\Gamma$ on $\X_t$, we denote the image of $[\Gamma]\in H_2(\X_t,\Z)$ under the map (\ref{singhom}) by $\sum_{i,j}c_{i,j}[\gamma_i]\otimes [\delta_j]$, where $\gamma_i$ and $\delta_j$ are 1-cycles on $C_t$ and $F$, respectively.
Then we have 
\begin{equation*}
\int_{\Gamma}\omega_t = \langle[\Gamma],[\omega_t] \rangle = \frac{1}{4}\langle[\Gamma],\phi([pr_1^*(\theta_t)\wedge pr_2^*(\mu)])\rangle = \dfrac{1}{4}\sum_{i,j}c_{i,j}\left(\int_{\gamma_i}\theta_t\right)\left(\int_{\delta_j}\mu\right).
\end{equation*}
Since the integration of $\mu$ is constant with respect to $t$, each period function with respect to $\omega$ can be expressed by a $\C$-linear combination of the period functions with respect to $\theta_t$.
Furthermore, by the relation (\ref{psipullback}), the period functions with respect to $\theta_t$ coincide with the period functions with respect to $\eta_\lambda$.
It is well-known (cf. \cite{CMP} p. 11) that each period function with respect to $\eta_\lambda$ is a $\Z$-linear combination of the following multivalued functions
\begin{equation*}
P_1(\lambda) = \int_0^1\frac{dx}{\sqrt{x(x-1)(x-\lambda)}}\quad\text{and}\quad P_2(\lambda) = \int_1^\infty\frac{dx}{\sqrt{x(x-1)(x-\lambda)}}.
\end{equation*}
Thus, any period function with respect to $\omega$ can be expressed by a $\C$-linear combination of $P_1(\lambda)$ and $P_2(\lambda)$.
These functions satisfy the following Gauss hypergeometric differential equation of type $(1/2,1/2,1)$ (cf. \cite{CMP}, p. 15).
\begin{equation*}
\lambda(1-\lambda)\dfrac{d^2P}{d\lambda^2}+(1-2\lambda)\dfrac{dP}{d\lambda}-\frac{1}{4}P=0
\end{equation*}
If we rewrite this equation with respect to $t=-\frac{1}{2}\lambda$, we see that the differential operator $\Dc_t$ annihilates every period function with respect to $\omega$.
\end{proof}

Since $\Dc_t$ annihilates all sections of $\Pc_\omega$, the differential operator $\Dc_t$ factors the sheaf $\Qc_\omega$.
Thus we have the following morphism of sheaves of abelian groups.
\begin{equation}\label{DCQ}
\Dc_t: \Qc_\omega \lra \O_T^\an
\end{equation}

\subsection{The transcendental regulator map}
In this section, we fix a point $t \in T$ such that $\sqrt{t}\in \R_{>0}$ and compute the image of $\xi_{t}$ under the transcendental regulator map.
Then we will get an integral representation of a multivalued holomorphic function representing $\nu_{\tr}(\xi)(\omega) \in \Gamma(T,\Qc_\omega)$.

Let $R_0$ (resp. $R_1$) be the point in $l_0\cap C_{t}$ (resp. $l_1\cap C_{t}$) such that its $z_1$-coordinate is positive.
Let $\Gamma \subset \R^2$ be the closed bounded subset bounded by the real curve $C_{t}(\R)$ and the line segments $\overline{PR_0},\overline{PR_1}$ (see Figure \ref{gfigure}).
Since we have the natural orientation on $\Gamma$ induced by $\R^2$, we regard $\Gamma$ as a 2-chain on $\P^2$.
We define a subset $\Gamma^\circ$ of $\Gamma$ by
\begin{equation*}
\Gamma^\circ = \left\{(z_1,z_2)\in \Gamma: z_2^4-f_t(z_1)>0\right\}.
\end{equation*}

\begin{figure}[h]
\centering
\includegraphics[width=12cm]{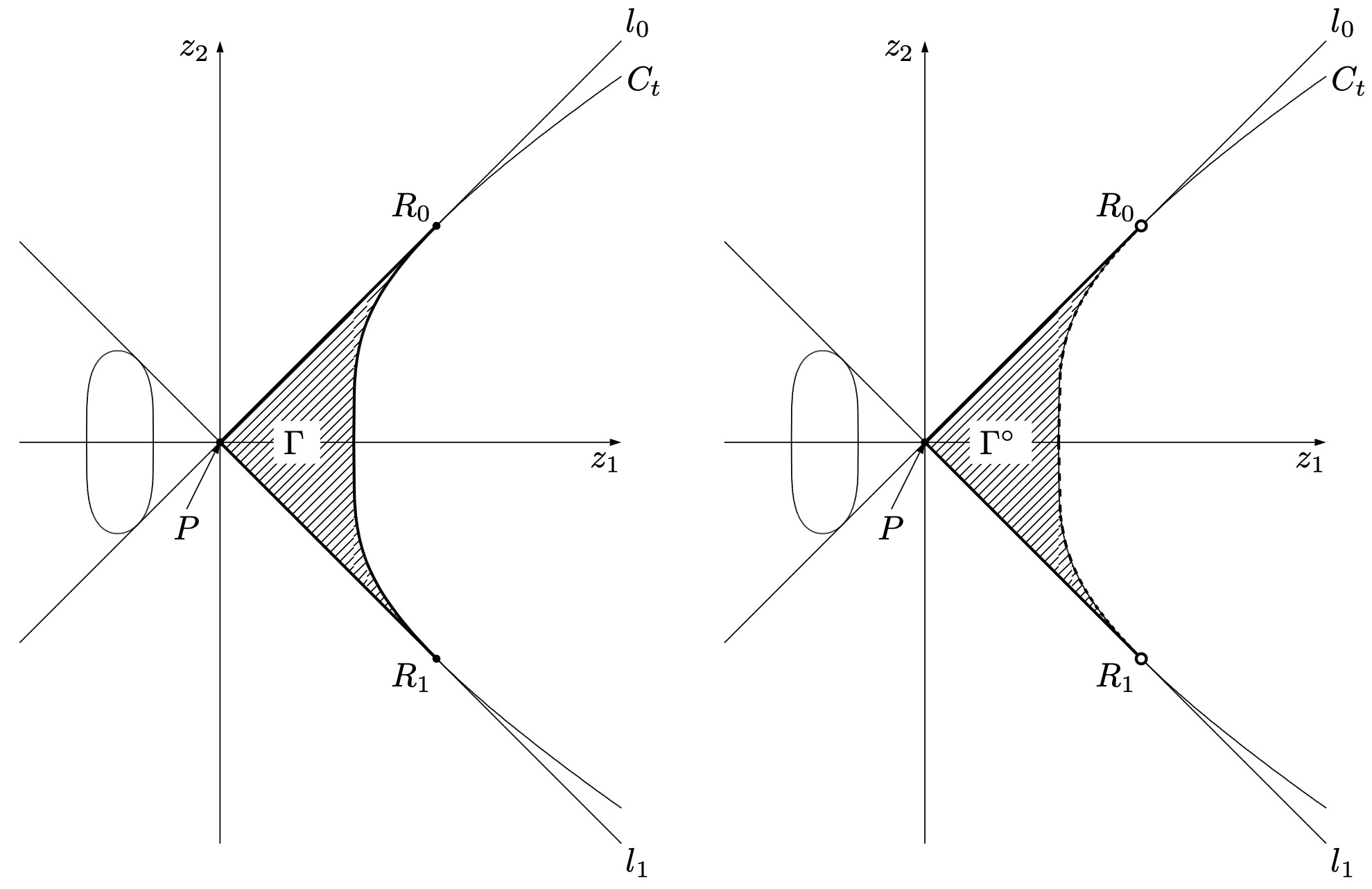}
\caption{$\Gamma$ and $\Gamma^\circ$ in $\P^2$}
\label{gfigure}
\end{figure}

\begin{prop}\label{2chainprop}
The image of $\xi_{t}$ under the transcendental regulator map is given as follows:
\begin{equation*}
\Phi(\xi_t)([\omega_t]) \equiv \int_{\Gamma^\circ}\frac{dz_1\wedge dz_2}{\left(z_2^4-f_{t}(z_1)\right)^{3/4}}  \mod \Pc(\omega_{t}),
\end{equation*}
where $\left(z_2^4-f_{t}(z_1)\right)^{3/4}$ denotes the cube of the positive 4th root of $z_2^4-f_{t}(z_1)\in \R_{>0}$.
\end{prop}
\begin{proof}
Recall that $\xi_{t}$ is represented by the formal sum $(l_{0,0},f_{0,0})+(l_{1,0},f_{1,0})$. 
For $i=0,1$, let $\gamma_{i}$ be the pull-back of the path $[\infty,0]$ (see Section 2.2) on $\P^1$ by $f_{i,0}:l_{i,0}\rightarrow \P^1$.
By the relation (\ref{ratfunction1}), $\gamma_0$ (resp. $\gamma_1$) is the path from $Q_0$ to $Q_2$ (resp. $Q_2$ to $Q_0$).
We will find a 2-chain associated with $\xi_t$, in other words, a 2-chain on $\X_{t}$ whose boundary coincides with $\gamma_0+\gamma_1$.

Let $\Gamma_+$ and $\Gamma_-$ be the closed subsets of the real surface $\X_{t}(\R)$ defined by
\begin{equation*}
\begin{aligned}
\Gamma_+ &= \left\{\left(z_1,z_2,w\right)\in \R^3: (z_1,z_2)\in \Gamma, w = \left(z_2^4-f_{t}(z_1)\right)^{1/4}\right\} \text{ and }  \\
\Gamma_- &= \left\{\left(z_1,z_2,w\right)\in \R^3: (z_1,z_2)\in \Gamma, w = -\left(z_2^4-f_{t}(z_1)\right)^{1/4}\right\}.
\end{aligned}
\end{equation*}
Then $\pi: \X_{t}\rightarrow \P^2$ induces the homeomorphisms between $\Gamma_+,\Gamma_-$ and $\Gamma$. 
By these homeomorphisms, we have natural orientations on $\Gamma_+$ and $\Gamma_-$, so we regard them as 2-chains on $\X_{t}$.
We will define the paths $\gamma_{0,\pm}, \gamma_{1,\pm}$ on $\del\Gamma_\pm$ as shown in Figure \ref{gpmfigure}.

\begin{figure}[h]
\centering
\includegraphics[width=6.8cm]{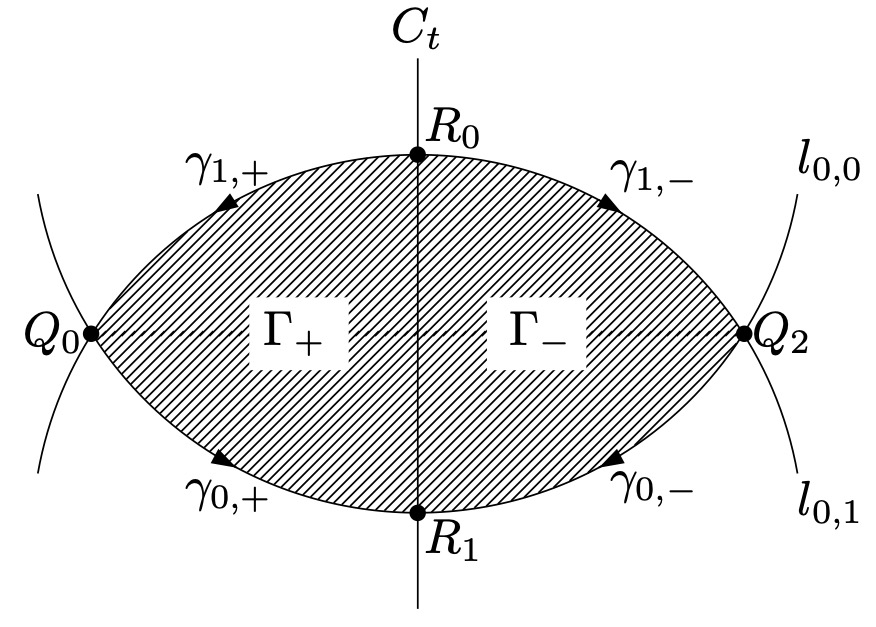}
\caption{The 2-chains $\Gamma_+$ and $\Gamma_-$ on $\X_t$}
\label{gpmfigure}
\end{figure} 

Then we have
\begin{equation}\label{boundary}
\del(\Gamma_+-\Gamma_-) = (\gamma_0+\gamma_1) + (\gamma_{0,+}-\gamma_{0,-}-\gamma_0) + (\gamma_{1,+}-\gamma_{1,-}-\gamma_{1}).
\end{equation}
For $i=0,1$, the 1-chain $\gamma_{i,+}-\gamma_{i,-}-\gamma_{i}$ on $l_{i,0}$ satisfies $\del(\gamma_{i,+}-\gamma_{i,-}-\gamma_{i})=0$, so this is a 1-cycle.
Since $l_{i,0}$ is isomorphic to $\P^1$ and $H_1(\P^1,\Z)=0$, this is also a 1-boundary.
Thus we can find a 2-chain $\Gamma_i$ on $l_{i,0}$ such that $\del\Gamma_i = \gamma_{i,+}-\gamma_{i,-}-\gamma_{i}$.
By the relation (\ref{boundary}), the 2-chain $\Gamma_+-\Gamma_--\Gamma_0-\Gamma_1$ is a 2-chain associated with $\xi$.
By the formula (\ref{transregval}), we have
\begin{equation*}
\Phi(\xi_t)([\omega_t]) \equiv \int_{\Gamma_+}\omega_{t} - \int_{\Gamma_-}\omega_{t} - \int_{\Gamma_0}\omega_{t} - \int_{\Gamma_1}\omega_{t} \mod \Pc(\omega_{t}).
\end{equation*}
Since $\Gamma_0$ and $\Gamma_1$ are 2-chains on curves, the restrictions of $\omega_t$ on them vanish.
Furthermore, since $\Gamma_- = \sigma^2(\Gamma_+)$ and $(\sigma^2)^*\omega_t = -\omega_t$, we have
\begin{equation*}
\int_{\Gamma_+}\omega_{t} - \int_{\Gamma_-}\omega_{t} - \int_{\Gamma_0}\omega_{t} - \int_{\Gamma_1}\omega_{t}  =  \int_{\Gamma_+}\omega_{t} -  \int_{\Gamma_-}\omega_{t} = 2\int_{\Gamma_+}\omega_{t}.
\end{equation*}
Since the integration of $\omega_t$ on $(\pi|_{\Gamma_+})^{-1}(\Gamma\setminus\Gamma^\circ)$ is $0$, we have
\begin{equation*}
2\int_{\Gamma_+}\omega_{t} =  2\int_{(\pi|_{\Gamma_+})^{-1}(\Gamma^\circ)}\omega_t= 2\int_{\Gamma^\circ}\frac{dz_1\wedge dz_2}{\left(z_2^4-f_{t}(z_1)\right)^{3/4}}.
\end{equation*}
Combining the above equations, we get the result.
\end{proof}

For $t\in \R_{>0}$, we define the function $G(t)$ by the improper integral
\begin{equation*}
G(t) = 2\int_{\Gamma^\circ}\frac{dz_1\wedge dz_2}{\left(z_2^4-f_{t}(z_1)\right)^{3/4}}.
\end{equation*}
Note that the integration converges by Proposition \ref{2chainprop}.
The following technical proposition is the final piece in the proof of the main theorem.
\begin{prop}\label{diffeqnG}
We can extend $G(t)$ to the multivalued holomorphic function on $T$ satisfying the differential equation
\begin{equation}\label{diffeqnGeqn}
\Dc_t(G(t))= \frac{1}{2\sqrt{t}(t+1)},
\end{equation}
where $\Dc_t$ is the Picard-Fuchs differential operator in Proposition \ref{PFop}.
\end{prop}
Using this proposition, we can show the following.

\begin{thm}\label{tnonzero}
For a very general $t\in T$, $\Phi(\xi_t)$ is non-torsion.
\end{thm}
\begin{proof}
By Proposition \ref{2chainprop}, we see that 
\begin{equation}\label{coincideeqn}
\ev_{t}(\nu_\tr(\xi|_T)(\omega)) = \Phi(\xi_t)([\omega_t]) \equiv G(t) \mod \Pc(\omega_t).
\end{equation}
holds for $t\in T$ such that $\sqrt{t}\in \R_{>0}$.
By Lemma \ref{importantlemma}, $\nu_{\tr}(\xi|_T)(\omega)$ is represented by the multivalued holomorphic function $G(t)$.
Then we have $\Dc_t(\nu_{\tr}(\xi|_T)(\omega)) = \Dc_t(G(t))$. 
Note that we regard $\Dc_t$ on the left-hand side as the map $\Gamma(T,\Qc_\omega)\rightarrow \Gamma(T,\O_T^\an)$ induced by the morphism (\ref{DCQ}).
By Proposition \ref{diffeqnG}, $\Dc_t(G(t))\neq 0$, so $\nu_{\tr}(\xi|_T)(\omega)$ is non-torsion.
Then by Lemma \ref{importantlemma}, we have the result.
\end{proof}

Finally, we can prove the main theorem.
\begin{proof}[Proof of Theorem \ref{mainthm}]
By Theorem \ref{tnonzero}, there exists a point $t_0\in T$ such that $\Phi(\xi_{t_0})$ is non-torsion.
Then $\nu_{\sqrt{-1}}(\xi_t)$ is non-torsion for such a $t$ by Proposition \ref{reduce1}.
By Proposition \ref{sectionlem}, $\nu_{\sqrt{-1}}(\xi)$ is non-torsion, thus by Proposition \ref{nonzeroprop}, we have the result.
\end{proof}

\subsection{Proof of Proposition \ref{diffeqnG}}
Finally, we will prove Proposition \ref{diffeqnG} and complete the proof of the main theorem.
If the differential equation (\ref{diffeqnGeqn}) holds for $t\in T$ such that $\sqrt{t}\in \R_{>0}$, it is clear that we can extend $G(t)$ holomorphically on $T$, so we will prove (\ref{diffeqnGeqn}) for $\sqrt{t}\in \R_{>0}$.

\begin{figure}[h]
\centering
\includegraphics[width=7.2cm]{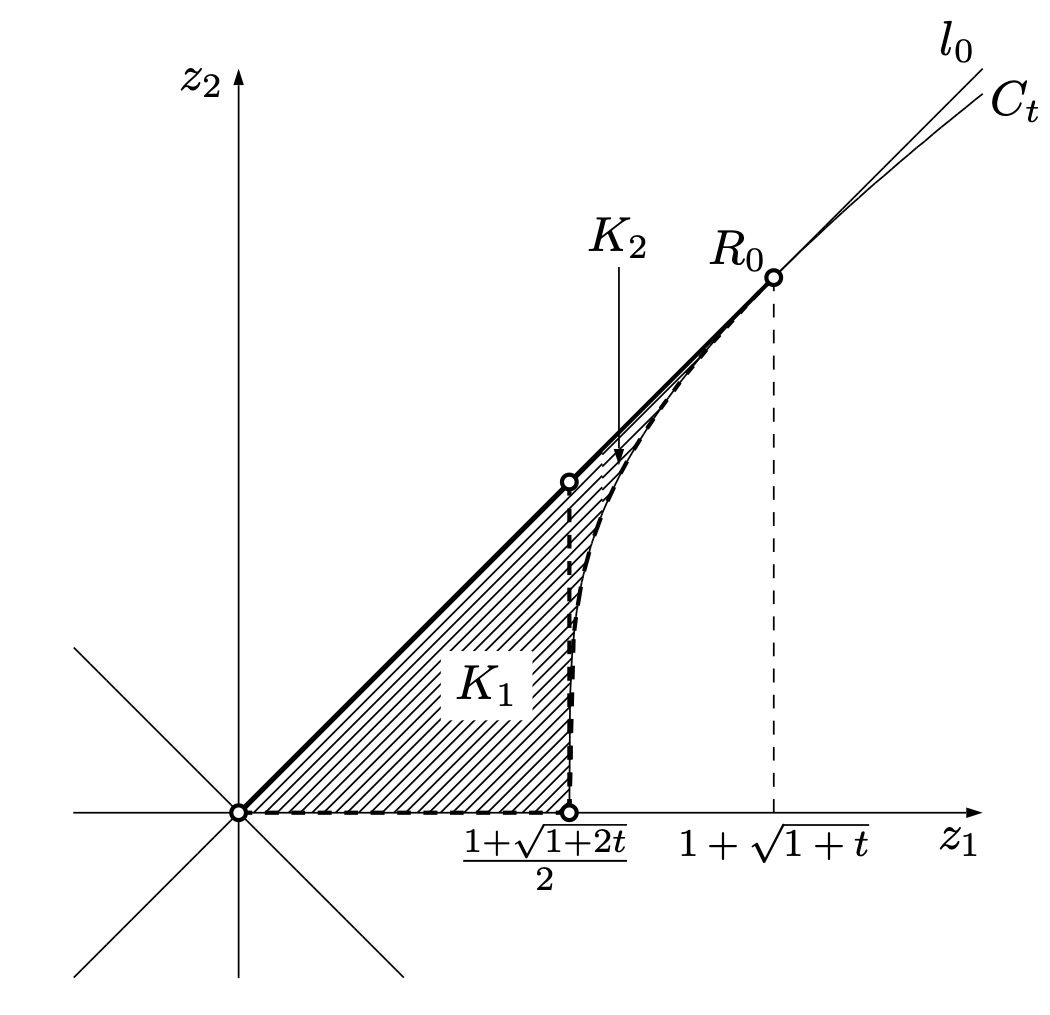}
\caption{$K_1$ and $K_2$}
\label{kfigure}
\end{figure} 

Let $K_1$ and $K_2$ be subsets of $\Gamma^\circ$ defined by 
\begin{equation*}
\begin{aligned}
K_1 & = \{(z_1,z_2)\in \Gamma^\circ : z_2>0, z_1 < \alpha\} \text{ and } \\
K_2 & = \{(z_1,z_2)\in \Gamma^\circ : z_2>0, z_1 > \alpha\},
\end{aligned}
\end{equation*}
where $\alpha=(1+\sqrt{1+2t})/2$ is the positive root of $f_t(z)=0$ (Figure \ref{kfigure}).
Since the integrand of $G(t)$ is invariant under $z_2\mapsto -z_2$, we have
\begin{equation}\label{GtoK}
G(t) = 4\int_{K_1\sqcup K_2}\frac{dz_1\wedge dz_2}{\left(z_2^4-f_{t}(z_1)\right)^{3/4}}.
\end{equation}

We consider the rational map $\tau$ defined by
\begin{equation*}
\begin{tikzcd}[row sep = tiny]
\tau:&[-35pt] \P^2\arrow[r,dashrightarrow] &\P^1\times \P^1 \\
& (z_1,z_2)\aru \arrow[r,mapsto] & \left(x,p\right) = \left(\dfrac{2z_1^2}{2z_1+t}, \dfrac{z_2^4}{z_2^4-f_{t}(z_1)}\right).\aru
\end{tikzcd}
\end{equation*}
The rational map $\tau$ fits into the following diagram.
\begin{equation}\label{diag}
\begin{tikzcd}
(x,y,u_1,u_2)\arrow[d,mapsto]\arrow[r,phantom,"\in"]&[-20pt]E_t \times F \arrow[d] & C_t\times F \arrow[r,dashrightarrow,"\varphi"] \arrow[l,"\psi\times \id"']& \X_t\arrow[d,"\pi"] \\
(x,p) = \left(x,\dfrac{u_1^4}{u_2^4}\right)\arrow[r,phantom,"\in"]&\P^1\times \P^1 && \P^2 \arrow[ll, "\tau",dashrightarrow]
\end{tikzcd}
\end{equation} 
where $\psi$ is the morphism defined by (\ref{psidef}).
Let $\triangle_1$ and $\triangle_2$ be the subsets of $\P^1\times \P^1$ defined by
\begin{equation*}
\begin{aligned}
\triangle_1 & = \left\{(x,p)\in \R^2 :0<  p\le \left(\frac{x}{2-x}\right)^2\text{ and } 0<x<1\right\} \text{ and } \\
\triangle_2 & = \left\{(x,q)\in \R^2 :0<q\le \left(\frac{2-x}{x}\right)^2 \text{ and } 1<x<2 \right\},
\end{aligned}
\end{equation*}
where $q = 1/p$ (Figure \ref{triangle}).
Then $\tau$ induces the orientation-preserving homeomorphisms $K_1\simeq \triangle_1$ and $K_2\simeq \triangle_2$.

\begin{figure}[h]
\centering
\includegraphics[width=11.5cm]{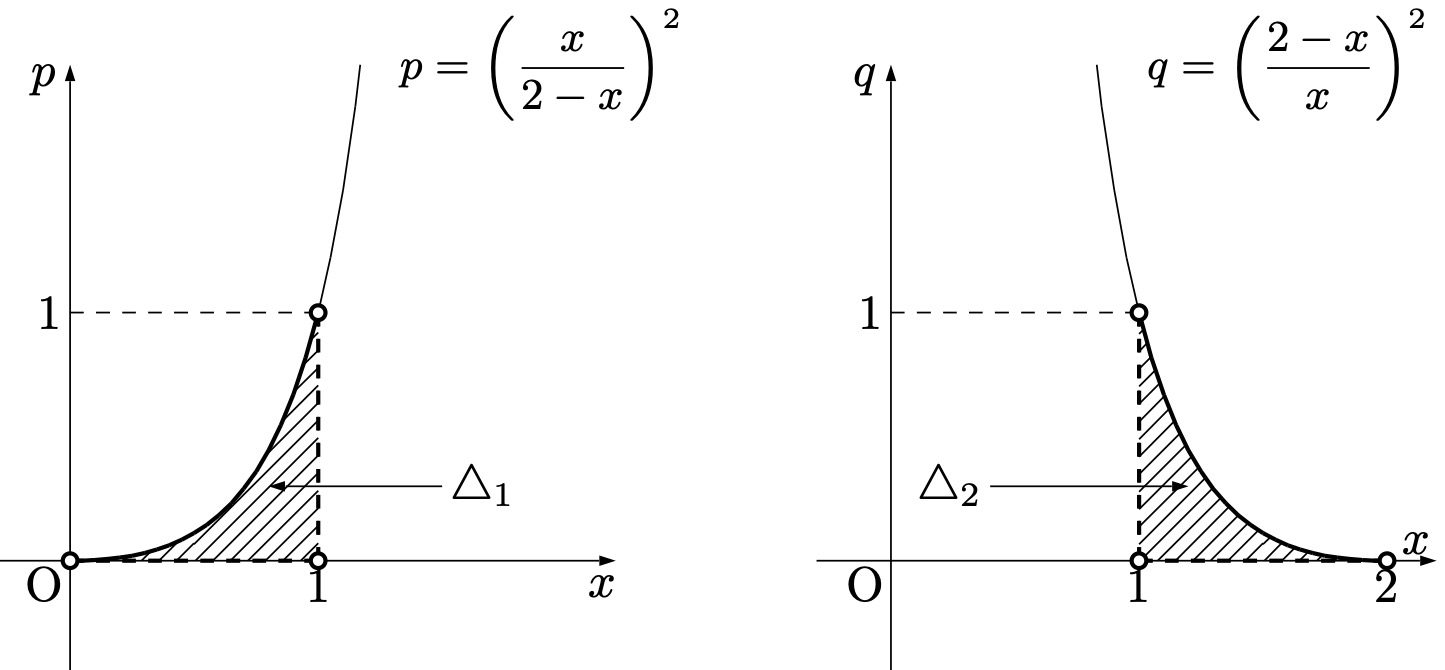}
\caption{$\triangle_1$ and $\triangle_2$}
\label{triangle}
\end{figure} 

Using the relations (\ref{formrelation})(\ref{psipullback}) and the diagram (\ref{diag}), we see that the 2-form $\dfrac{dz_1\wedge dz_2}{\left(z_2^4-f_{t}(z_1)\right)^{3/4}}$ on $K_1\sqcup K_2$ is the pull-back of the 2-form \\
$\dfrac{dx\wedge dp}{8x^{\frac{1}{2}}(1-x)^{\frac{1}{2}}(x+2t)^{\frac{1}{2}}p^{\frac{3}{4}}(1-p)^{\frac{1}{2}}}$ on $\triangle_1$ and 
$\dfrac{dx\wedge dq}{8x^{\frac{1}{2}}(x-1)^{\frac{1}{2}}(x+2t)^{\frac{1}{2}}q^{\frac{3}{4}}(1-q)^{\frac{1}{2}}}$ on $\triangle_2$.
By (\ref{GtoK}), $G(t)$ coincides with 
\begin{equation}\label{Ktotriangle}
\int_{\triangle_1}\dfrac{dxdp}{2x^{\frac{1}{2}}(1-x)^{\frac{1}{2}}(x+2t)^{\frac{1}{2}}p^{\frac{3}{4}}(1-p)^{\frac{1}{2}}} + \int_{\triangle_2}\dfrac{dxdq}{2x^{\frac{1}{2}}(x-1)^{\frac{1}{2}}(x+2t)^{\frac{1}{2}}q^{\frac{3}{4}}(1-q)^{\frac{1}{2}}}.
\end{equation}
Note that the domains of the integrations do not depend on $t$.

For $i\ge 1$, we will prove 
\begin{equation}\label{interchangerel}
\begin{aligned}
&\frac{d^i}{d t^i}\left(\int_{\triangle_1}\dfrac{dxdp}{2x^{\frac{1}{2}}(1-x)^{\frac{1}{2}}(x+2t)^{\frac{1}{2}}p^{\frac{3}{4}}(1-p)^{\frac{1}{2}}}\right) \\
&\qquad\qquad= \int_{\triangle_1}\frac{\del^i}{\del t^i}\left(\dfrac{1}{2x^{\frac{1}{2}}(1-x)^{\frac{1}{2}}(x+2t)^{\frac{1}{2}}p^{\frac{3}{4}}(1-p)^{\frac{1}{2}}}\right)dxdp.
\end{aligned}
\end{equation}
For any $t_0>0$, there exists a constant $M$ such that, for any $t$ which is sufficiently close to $t_0$, the absolute value of the integrand on the right-hand side of (\ref{interchangerel}) can be bounded by the integrable function $M/x^{\frac{1}{2}}(1-x)^{\frac{1}{2}}p^{\frac{3}{4}}(1-p)^{\frac{1}{2}}$ on $\triangle_1$.
Therefore, by Lebesgue's dominated convergence theorem, (\ref{interchangerel}) follows.
We have a similar result for the integral over $\triangle_2$.

In particular, we have
\begin{equation}\label{interchange}
\begin{aligned}
\Dc_t(G(t)) =  &\int_{\triangle_1}\Dc_t\left(
\frac{1}{2x^{\frac{1}{2}}(1-x)^{\frac{1}{2}}(x+2t)^{\frac{1}{2}}p^{\frac{3}{4}}(1-p)^{\frac{1}{2}}}
\right)dxdp   \\
&+\int_{\triangle_2}\Dc_t\left(\dfrac{1}{2x^{\frac{1}{2}}(x-1)^{\frac{1}{2}}(x+2t)^{\frac{1}{2}}q^{\frac{3}{4}}(1-q)^{\frac{1}{2}}}\right)dxdq.
\end{aligned}
\end{equation}
Moreover, by the relation 
\begin{equation*}
\Dc_t\left(
\dfrac{1}{x^{\frac{1}{2}}(x-1)^{\frac{1}{2}}(x+2t)^{\frac{1}{2}}}\right) = \dfrac{\del}{\del x}\left(\dfrac{x^\frac{1}{2}(x-1)^\frac{1}{2}}{2(x+2t)^{\frac{3}{2}}}\right),
\end{equation*}
the right-hand side of (\ref{interchange}) can be written as
\begin{equation}\label{potential}
\displaystyle\int_{\triangle_1} \dfrac{\del}{\del x}\left(\dfrac{-x^\frac{1}{2}(1-x)^\frac{1}{2}}{4(x+2t)^{\frac{3}{2}}p^{\frac{3}{4}}(1-p)^{\frac{1}{2}}}\right)dxdp + \displaystyle\int_{\triangle_2}\dfrac{\del}{\del x}\left(\dfrac{x^{\frac{1}{2}}(x-1)^{\frac{1}{2}}}{4(x+2t)^{\frac{3}{2}}q^{\frac{3}{4}}(1-q)^{\frac{1}{2}}}\right)dxdq
\end{equation}

Finally, for $\varepsilon>0$, let $\triangle_{1,\varepsilon}$ be the closed subset of $\triangle_1$ defined by
\begin{equation*}
\triangle_{1,\varepsilon} = \left\{(x,p)\in \R^2:\varepsilon\le p\le \left(\frac{x}{2-x}\right)^2, \varepsilon\le x\le 1-\varepsilon\right\}.
\end{equation*}
We have
\begin{equation*}
\begin{aligned}
&\displaystyle\int_{\triangle_1} \dfrac{\del}{\del x}\left(\dfrac{-x^\frac{1}{2}(1-x)^\frac{1}{2}}{4(x+2t)^{\frac{3}{2}}p^{\frac{3}{4}}(1-p)^{\frac{1}{2}}}\right)dxdp \\
&=  \lim_{\varepsilon\rightarrow 0}\displaystyle\int_{\triangle_{1,\varepsilon}} \dfrac{\del}{\del x}\left(\dfrac{-x^\frac{1}{2}(1-x)^\frac{1}{2}}{4(x+2t)^{\frac{3}{2}}p^{\frac{3}{4}}(1-p)^{\frac{1}{2}}}\right)dxdp \\
&= \lim_{\varepsilon \rightarrow 0} \int_{\triangle_{1,\varepsilon}} d\left(\dfrac{-x^\frac{1}{2}(1-x)^\frac{1}{2}dp}{4(x+2t)^{\frac{3}{2}}p^{\frac{3}{4}}(1-p)^{\frac{1}{2}}}\right)
= \lim_{\varepsilon \rightarrow 0} \int_{\del\triangle_{1,\varepsilon}} \dfrac{-x^\frac{1}{2}(1-x)^\frac{1}{2}dp}{4(x+2t)^{\frac{3}{2}}p^{\frac{3}{4}}(1-p)^{\frac{1}{2}}}.
\end{aligned}
\end{equation*}
We use the Stokes theorem in the last equality.
Since the 1-form 
\begin{equation*}
 \dfrac{-x^\frac{1}{2}(1-x)^\frac{1}{2}dp}{4(x+2t)^{\frac{3}{2}}p^{\frac{3}{4}}(1-p)^{\frac{1}{2}}}
\end{equation*}
on $\del\triangle_{1,\varepsilon}$ vanishes or converges to $0$ when $\varepsilon\rightarrow 0$ except on the curve $p=(x/(2-x))^2$,
\begin{equation*}
\lim_{\varepsilon \rightarrow 0} \int_{\del\triangle_{1,\varepsilon}} \dfrac{-x^\frac{1}{2}(1-x)^\frac{1}{2}dp}{4(x+2t)^{\frac{3}{2}}p^{\frac{3}{4}}(1-p)^{\frac{1}{2}}}  = \int_{x=0}^{x=1}\frac{dx}{2(x+2t)^{\frac{3}{2}}(2-x)^{\frac{1}{2}}}.
\end{equation*}
Similarly, we have 
\begin{equation*}
\displaystyle\int_{\triangle_2}\dfrac{\del}{\del x}\left(\dfrac{x^{\frac{1}{2}}(x-1)^{\frac{1}{2}}}{4(x+2t)^{\frac{3}{2}}q^{\frac{3}{4}}(1-q)^{\frac{1}{2}}}\right)dxdq = \int_{x=1}^{x=2}\frac{dx}{2(x+2t)^{\frac{3}{2}}(2-x)^{\frac{1}{2}}}.
\end{equation*}
Combining the above equations, we have
\begin{equation*}
\Dc_t(G(t)) =  \int_{x=0}^{x=2}\frac{dx}{2(x+2t)^{\frac{3}{2}}(2-x)^{\frac{1}{2}}} =\int_{u=0}^{u=1/\sqrt{t}}\dfrac{du}{2(t+1)}= \frac{1}{2\sqrt{t}(t+1)}
\end{equation*}
where we use the variable transformation $u=\sqrt{(2-x)/(x+2t)}$.
Thus we obtain the differential equation.
Note that if we rewrite the equation by $\lambda = -2t$, we get the differential equation (\ref{inhomogPF}) in the introduction.

\bibliographystyle{plain}
\bibliography{reference}

\begin{thebibliography}{99}

\bibitem[Blo86]{bloch} 
Bloch, S. 
\textit{Algebraic cycles and higher K-theory.}
Adv. in Math.  \textbf{61}  (1986),  no. 3, 267--304.

\bibitem[CMP02]{CMP}
Carlson, J.; M\"uller-Stach, S.~J.; Peters, C.
\textit{Period mappings and period domains.} 
Cambridge Studies in Advanced Mathematics, 85. Cambridge University Press, Cambridge, 2003. xvi+430 pp. ISBN: 0-521-81466-9

\bibitem[CDKL16]{CDKL16}
Chen, X.; Doran, C.; Kerr, M.; Lewis, J. D.
\textit{Normal functions, Picard-Fuchs equations, and elliptic fibrations on K3 surfaces.}
J. Reine Angew. Math.  \textbf{721}  (2016). 43--79.

\bibitem[CL05]{CL}
Chen, X.; Lewis, J. D.
\textit{The Hodge-D-conjecture for K3 and abelian surfaces.}
J. Algebraic Geom. \textbf{14} (2005), no.2, 213--240.

\bibitem[dAM08]{dAMS08}
del Angel, P. L.; M\"uller -Stach, S.~J.
\textit{Differential equations associated to families of algebraic cycles.}
Ann. Inst. Fourier \textbf{58} (2008), 2075--2085.

\bibitem[KLM06]{KLM}
Kerr, M.; Lewis, J. D.; M\"uller-Stach, S.~J.
\textit{The Abel-Jacobi map for higher Chow groups}
Compos. Math. \textbf{142} (2006), no. 2, 374--396.

\bibitem[Kon00]{Kondo}
Kond\= o, S.
\textit{A complex hyperbolic structure of the moduli space of curves of genus three.}
J. reine angew. Math., \textbf{525} (2000), 219--232.

\bibitem[Kon11]{Kondo2}
 Kond\= o, S.
\textit{Moduli of plane quartics, G\"opel invariants and Borcherds products.}
Int. Math. Res. Not. IMRN. \textbf{12} (2011), 2825--2860.
 
\bibitem[Lev88]{Le} 
Levine, M.
\textit{Localization on singular varieties.}
Invent. Math.  \textbf{91}  (1988),  no. 3, 423--464.

\bibitem[MS23]{MS23}
Ma, S.; Sato, K.
\textit{Higher Chow cycles on some $K3$ surfaces with involution.}
arXiv:2309.16132.

\bibitem[M\"ul98]{MS2}
M\"uller-Stach, S.~J. 
\textit{Algebraic cycle complexes: basic properties.} 
The arithmetic and geometry of algebraic cycles (Banff, AB, 1998), 285--305, NATO Sci. Ser. C Math. Phys. Sci., \textbf{548}, Kluwer Acad. Publ., Dordrecht, 2000.

\bibitem[Sat24]{Sat24}
Sato, K. 
\textit{Higher Chow cycles on a family of Kummer surfaces.} 
arXiv:2211.16109, to appear in Canad. J. Math.

\bibitem[SK79]{SK79}
Shioda, T.; Katsura, T.
\textit{On Fermat varieties.}
T\^ohoku J. Math., \textbf{31} (1979), 97--115.

\bibitem[Sre14]{Sre14}
Sreekantan, R.
\textit{Higher Chow cycles on Abelian surfaces and a non-Archimedean analogue of the Hodge-$\mathcal{D}$-conjecture.} 
Compos. Math. \textbf{150} (2014), no.4, 691--711. 

\bibitem[Sre24]{Sre24}
Sreekantan, R.
\textit{Indecomposable motivic cycles on K3 surfaces of degree 2},
arXiv:2401.01052.

\bibitem[Voi02]{Voi02}
Voisin, C.
\textit{Hodge theory and complex algebraic geometry. I.}
Translated from the French by Leila Schneps. Reprint of the 2002 English edition.Cambridge Studies in Advanced Mathematics, 76. Cambridge University Press, Cambridge, 2007. x+322 pp. ISBN: 978-0-521-71801-1

\end{thebibliography}

\end{document}